\documentclass[12pt,a4paper]{article}
\usepackage{bm}
\usepackage{url}
\usepackage{chngcntr}
\usepackage{amsmath}
\usepackage{amssymb}
\usepackage[ruled,vlined]{algorithm2e}
\usepackage{amsthm}
\usepackage{float}
\usepackage{setspace}
\usepackage{graphicx}
\usepackage[nottoc,notlot,notlof]{tocbibind}
\usepackage{authblk}
\usepackage[top=2cm, bottom=2cm, left=2cm, right=2cm]{geometry}
\usepackage{mathtools}
\usepackage{array}
\usepackage{lmodern,blindtext}
\usepackage{float}
\floatstyle{plaintop}
\restylefloat{table}

\input xy
\xyoption{all}

\newtheorem{lemma}{Lemma}[section]
\newtheorem{corollary}{Corollary}[section]
\newtheorem{theorem}{Theorem}[section]

\theoremstyle{definition}
\newtheorem{remark}{Remark}[section]
\newtheorem{definition}{Definition}[section]
\newtheorem{example}{Example}[section]

\newcommand{\N}{\mathbb{N}}
\DeclareMathOperator*{\argmax}{arg\,max}

\counterwithout{equation}{section}

\begin{document}
\pagenumbering{arabic}
\setcounter{page}{1}

\title{On the coefficients of the distinct monomials in the expansion of $x_1(x_1+x_2)\cdots(x_1+x_2+\cdots+x_n)$}

\author{Sela Fried \thanks{The author is a postdoctoral fellow in the Department of Computer Science at the Ben-Gurion University of the Negev. Research Supported by the Israel Science Foundation (ISF) through grant No.
1456/18 and European Research Council Grant number: 949707.}}
\date{} 
\maketitle

\begin{abstract}
We initiate the study of the coefficients of the distinct monomials in the expansion of the multivariate polynomials $x_1(x_1+x_2)\cdots(x_1+x_2+\cdots+x_n), n\in\N$. In particular we obtain several results regarding their maximal coefficients.
\end{abstract} 

\section{Introduction}
Let $n\in\N$ and let $x_1,\ldots,x_n$ be indeterminates. It is well known that among the  multitude of their combinatorial interpretations, the Catalan numbers also count the distinct monomials in the expansion of the multivariate polynomials $$p_n=x_1(x_1+x_2)\cdots(x_1+x_2+\cdots+x_n),\;\;n\in\N,$$ a result that goes back at least to \cite{shallit1982e2972}. It seems that despite the naturalness of these polynomials, the coefficients of their distinct monomials have not been given attention by the community. In particular, to the best of our knowledge, a closed formula for their maximum is not known. In this work we initiate the study of these coefficients and accomplish the following: Upon establishing several elementary properties including a closed formula for them (Lemma \ref{lem; coef}), we show that, for every $n\in\N$, a maximal coefficient in the expansion of $p_n$ is attained at a monomial belonging to a certain set $\mathcal{M}_n$ that has the same cardinality as the set of all partitions of $n$ with distinct parts (Theorem \ref{thm; 1} and Lemma \ref{lem; 199}). Finally, we provide an algorithm (Algorithm \ref{IR}) that successively (and greedily) generates a sequence of monomials $(r_n)_{n\in\N}$ such that $r_n\in \mathcal{M}_n$ for every $n\in\N$ and conjecture that the corresponding coefficients are actually maximal.

\bigskip

Our interest in the polynomials $p_n, n\in\N$ was triggered during our work \cite{fried2021restrictiveness} on the restrictiveness of  stochastic orders where we established a closed formula (\cite[Lemma 3.6]{fried2021restrictiveness}) for the probability that a random probability distribution that is drawn uniformly from the probability $n$-simplex is greater than a fixed probability distribution with respect to the usual stochastic order. The formula involves a sum over all  distinct monomials of $p_{n-1}$. Thus, the distinct monomials of $p_n,n\in\N$ have a direct application in probability theory.

\section{Main results}

This section consist of two parts: In the first we establish several elementary properties of the coefficients of $p_n, n\in\N$ including a closed formula for them and in the second we address the problem of finding the maximal coefficients of $p_n, n\in\N$. Before we begin, let us, for completeness, prove the claim stated in the introduction that the distinct monomials in the expansion of $p_n, n\in\N$ are counted by the Catalan numbers which have the explicit formula $C_n = \frac{1}{n+1}\binom{2n}{n}$ for every $n\in\N_0$ where $\N_0 = \N\cup\{0\}$ (e.g. \cite[Theorem 1.4.1]{stanley2015catalan}). To simplify formulations throughout this work, whenever we refer to a monomial of $p_n$, we mean a monomial in the expansion of $p_n$ after combining like terms. In particular, in our terminology, the monomials of $p_n$ are distinct. Likewise, whenever we refer to a coefficient (of a monomial) of $p_n$, we mean the coefficient (of the monomial) after expanding $p_n$ and combining like terms. Finally, unless stated otherwise, $n\in\N$.

\begin{lemma}\label{lem; 1}
Let $\mathcal{P}_n$ denote the set of all monomials of $p_n$. Then $|\mathcal{P}_n|=C_n$.
\end{lemma}

\begin{proof}
The Catalan numbers have several fundamental interpretations. Thus, it suffices to construct a bijection between $\mathcal{P}_n$ and a set consisting of the elements of one of these interpretations. Following \cite{stanley2015catalan}, we shall prove that there is a bijection between $\mathcal{P}_n$ and the set $\mathcal{T}_n$ of plane trees with $n+1$ vertices (cf. \cite[p. 6 and Theorem 1.5.1]{stanley2015catalan}). The bijection goes through two auxiliary sets $\mathcal{A}_n$ and $\mathcal{B}_n$.

Let
$$\mathcal{A}_n = \left\{(a_1,\ldots,a_n)\in \N_0^n\;|\;\sum_{i={k+1}}^n a_i\leq n-k, \;\forall 1\leq k\leq n-1, \;\sum_{i=1}^n a_i=n\right\}.$$ We shall show that there are two maps $\Theta\colon \mathcal{A}_n\to \mathcal{P}_n$ and $\Phi\colon \mathcal{P}_n\to \mathcal{A}_n$ such that $\Theta\circ\Phi=\text{id}_{\mathcal{P}_n}$ and $\Phi\circ\Theta=\text{id}_{\mathcal{A}_n}$. For $(a_1,\ldots,a_n)\in \mathcal{A}_n$ let $$\Theta((a_1,\ldots,a_n))=x_1^{a_1}\cdots x_n^{a_n}.$$ The proof that $x_1^{a_1}\cdots x_n^{a_n}\in \mathcal{P}_n$ relies on the following observation: Suppose $i_1,\ldots,i_n\in\{1,\ldots,n\}$ are such that $1\leq i_k\leq k$ for every $1\leq k\leq n$. Then $x_{i_1}\cdots x_{i_n}\in P_n$. We define $i_1,\ldots,i_n$ as follows: First, if $a_n=1$ then we set $i_n=n$. Now, suppose we have already defined $i_{n+1-\sum_{i=k}^n a_i},\ldots,i_n$ for some $1<k\leq n$ such that $1\leq i_r\leq r$ for every $n+1-\sum_{i=k}^n a_i\leq r\leq n$. Set $i_{n+1-\sum_{i=k-1}^n a_i}=\cdots= i_{n-\sum_{i=k}^n a_i}=k-1$. Since $(a_1,\ldots,a_n)\in\mathcal{A}_n$, we have $$\sum_{i=k-1}^na_i\leq n+2-k\iff n+1-\sum_{i=k-1}^na_i\geq k-1.$$ It follows that $1\leq i_r\leq r$ for every $n+1-\sum_{i=k-1}^n a_i\leq r\leq n$.
    
In the other direction, for $x_1^{a_1}\cdots x_n^{a_n}\in \mathcal{P}_n$ we let $\Phi(x_1^{a_1}\cdots x_n^{a_n}) = (a_1,\ldots,a_n)$. To see that $(a_1,\ldots,a_n)\in \mathcal{A}_n$, first notice that $a_n\leq 1 = n-(n-1)$. Suppose now that we have already shown that $\sum_{i=k+1}^n a_i\leq n-k$ for some $1<k\leq n-1$. By definition of $p_n$, the indeterminate $x_k$ can appear in at most $n-k+1$ places, of which $\sum_{i=k+1}^n a_i$ places are already taken. Thus, $$a_k\leq n-k+1-\sum_{i=k+1}^n a_i\iff \sum_{i=k}^n a_i\leq n-(k-1).$$ 

Consider now the set $$\mathcal{B}_n =\left\{(b_1,\ldots,b_n)\in\N_0\cup\{-1\}\;|\;\sum_{i=1}^kb_i\geq0,\;\forall 1\leq k\leq n-1, \sum_{i=1}^nb_i=0\right\}.$$ 
One verifies immediately that
the map  $(a_1,\ldots,a_n)\mapsto (a_1-1,\ldots,a_n-1)$ gives a bijection between $\mathcal{A}_n$ and $\mathcal{B}_n$. 

\iffalse
Indeed, let $1\leq k\leq n-1$. Then 
\begin{align}
\sum_{i=k+1}^na_i\leq n-k\iff& n\leq \sum_{i=1}^ka_i+n-k \iff \sum_{i=1}^k(a_i-1) \geq 0 \text{ and}\nonumber\\ \sum_{i=1}^n a_i = n \iff & \sum_{i=1}^n (a_i-1)=0.\nonumber
\end{align}\fi
It remains to show that there is a bijection between $\mathcal{B}_n$ and $\mathcal{T}_n$. By \cite[82 on p. 71]{stanley2015catalan}, the following procedure provides such a bijection: Perform a depth-first search through a plane tree with $n+1$ vertices and every time a vertex is encountered for the first time, record one less than its number of children, except that the last vertex is ignored. We prove by induction on $n$ that the resulting sequence belongs to $\mathcal{B}_n$. For $n=1$ the plane tree has $2$ vertices and the corresponding sequence is necessarily $(0)$ which obviously belongs to $\mathcal{B}_1$. Suppose the claim holds for plane trees with $n$ vertices and consider a plane tree $T$ with $n+1$ vertices together with the corresponding sequence $(b_1,\ldots,b_n)$. Deleting the last vertex from $T$ gives a plane tree $T'$ with $n$ vertices. Let $1\leq l\leq n$ be the index of the parent of the last vertex of $T$. We distinguish between two cases: 
 
\begin{enumerate}
    \item $l=n$. In this case $T$ must end with a sequence of three vertices that make a tree of depth two. Thus, $b_n=0$ and $(b_1,\ldots,b_{n-1})$ is the sequence corresponding to $T'$. Using the induction hypothesis and putting back $b_n=0$ we see that $(b_1,\ldots,b_n)\in\mathcal{B}_n$.
    \item $l<n$. In this case, the $n$th vertex of $T$ is a leaf and therefore $b_n=-1$. Now, $b_l\geq 0$ and  $(b_1,\ldots,b_{l-1},b_l-1,b_{l+1},\ldots,b_{n-1})$ is the sequence corresponding $T'$, which, by the induction hypothesis, belongs to $\mathcal{B}_{n-1}$. Putting back $b_n=-1$ and replacing $b_l-1$ with $b_l$ we obtain the original sequence and conclude that it belongs to $\mathcal{B}_n$.
\end{enumerate}

One can show, again by induction, that any $(b_1,\ldots,b_n)\in\mathcal{B}_n$ induces a unique plane tree with $n+1$ vertices. We leave the details out.
\end{proof}

\begin{remark}
It follows from the proof of Lemma \ref{lem; 1} that we may identify $\mathcal{P}_n$ with $\mathcal{A}_n$ and we shall exploit this equivalent representation freely throughout this work. In particular, we shall refer to elements of $\mathcal{A}_n$ as `monomials'.
\end{remark}

\subsection{Elementary properties of the coefficients of $p_n$}

Our first result is an explicit formula for the coefficients of $p_n$. We shall use the following notation:

\begin{definition}
For $(a_1,\ldots,a_n)\in \mathcal{A}_n$ we denote by $c_{(a_1,\ldots,a_n)}$ the corresponding coefficient.
\end{definition}

\begin{lemma}\label{lem; coef}
Let $(a_1,\ldots,a_n)\in \mathcal{A}_n$. Then $$c_{(a_1,\ldots,a_n)}=\prod_{k=1}^{n-1}\frac{n-k+1-\sum_{i=k+1}^{n}a_{i}}{a_{k}!}.$$
\end{lemma}

\begin{proof}
Beginning with $a_n$, we notice that $x_n$ can be taken solely from the last term of the product in the definition of $p_n$. Thus, there are $\binom{1}{a_n}$ possibilities to do that. Proceeding to $a_{n-1}$ we notice that $x_{n-1}$ can be taken only from the last two terms of the product, but not from those who contributed $x_n$. This gives $\binom{2-a_n}{a_{n-1}}$ possibilities. Continuing so until we reach $a_1$, we conclude that the number of possibilities to obtain $(a_1,\ldots,a_n)$ is given by

\begin{align}
\prod_{k=0}^{n-1}\binom{k+1-\sum_{i=n-k+1}^n a_i}{a_{n-k}}=&\prod_{k=0}^{n-1}\frac{\left(k-\sum_{i=n-(k-1)}^{n}a_{i}\right)!\left(k+1-\sum_{i=n-(k-1)}^{n}a_{i}\right)}{a_{n-k}!\left(k+1-\sum_{i=n-k}^{n}a_{i}\right)!}\nonumber\\=&\prod_{k=0}^{n-1}\frac{k+1-\sum_{i=n-(k-1)}^{n}a_{i}}{a_{n-k}!}\nonumber\\=&\prod_{k=1}^{n-1}\frac{n-k+1-\sum_{i=k+1}^{n}a_{i}}{a_{k}!}\nonumber.
\end{align}
\end{proof}

It is desirable, when writing down the expansion of the different $p_n,n\in\N$, to maintain consistency regarding the order of their terms. To this end we define an ordering on $\mathcal{A}_n$. This, of course, induces an ordering of the corresponding coefficients. A reasonable choice is lexicographic and in decreasing order:

\begin{definition}\label{def; 5}
Let $a=(a_1,\ldots,a_n),b=(b_1,\ldots,b_n)\in \mathcal{A}_n$ such that $a\neq b$ and let $k=\min\{1\leq i\leq n\;|\;a_i\neq b_i\}$. We write $a\prec b$ if $a_k > b_k$.
\end{definition} 

\begin{example}
In $\mathcal{A}_3$: $$(3,0,0)\prec(2,1,0)\prec(2,0,1)\prec(1,2,0)\prec(1,1,1).$$
\end{example}

\begin{example}
In the following ``triangle" we present the coefficients of $p_1,\ldots,p_5$ ordered according to Definition \ref{def; 5} (cf. OEIS A347917):

$$\begin{smallmatrix}
1 && & & & & & & & & & & & & & & & & & & & & & & & & & & & & & & & & & & & & & & & \\
1 &1 && & & & & & & & & & & & & & & & & & & & & & & & & & & & & & & & & & & & & & & \\ 
1 &2 &1 &1 &1 && & & & & & & & & & & & & & & & & & & & & & & & & & & & & & & & & & & & \\
1 &3 &2 &1 &3 &4 &2 &1 &1 &1 &2 &1 &1 &1 && & & & & & & & & & & & & & & & & & & & & & & & & & & \\
1 &4 &3 &2 &1 &6 &9 &6 &3 &3 &4 &2 &1 &1 &4 &9 &6 &3 &6 &8 &4 &2 &2 &1 &2 &1 &1 &1 &1 &3 &2 &1 &3 &4 &2 &1 &1 &1 &2 &1 &1 &1. 
\end{smallmatrix}$$
\end{example}

\begin{example}
Consider $c_{(n,0,\ldots,0)}$ and $c_{(1,\ldots,1)}$ which are, respectively, the first and last coefficients of $p_n$. It holds $$c_{(n,0,\ldots,0)} = c_{(1,\ldots,1)} = 1.$$ Indeed, \begin{align}c_{(n,0,\ldots,0)}=&\frac{1}{n!}\prod_{k=1}^{n-1}\left(n-k+1\right)=1 \;\;\text{ and}\nonumber\\c_{(1,\ldots,1)}=&\prod_{k=1}^{n-1}\left(n-k+1-(n-k)\right)=1.\nonumber\end{align}
\end{example}

\begin{lemma}
The coefficients of $p_{n+1}$ contain (at least) two copies of the coefficients of $p_n$.
\end{lemma}

\begin{proof}
Let $2\leq n\in\N$ and let $(a_1,\ldots,a_{n-1})\in A_{n-1}$. Clearly, $$(1,a_1,\ldots,a_{n-1}), (a_1,\ldots,a_{n-1},1)\in \mathcal{A}_n.$$ Now,
\begin{align}
    c_{(a_1,\ldots,a_{n-1}, 1)} = & \prod_{k=1}^{n-1}\frac{n-1-k+1-\sum_{i=k+1}^{n-1}a_{i}}{a_{k}!}\nonumber\\ = & \prod_{k=1}^{n-2}\frac{n-1-k+1-\sum_{i=k+1}^{n-1}a_{i}}{a_{k}!}\nonumber\\= & c_{(a_1,\ldots,a_{n-1})}\nonumber
    \end{align}
where in the second equality we used that $a_{n-1}!=1$. Similarly,
\begin{align}
    c_{(1,a_1,\ldots,a_{n-1})} =  & \frac{n-1+1-\sum_{i=1}^{n-1}a_{i}}{1!}\prod_{1=2}^{n-1}\frac{n-k+1-\sum_{i=k+1}^{n}a_{i-1}}{a_{k-1}!}\nonumber\\= &(n-(n-1))\prod_{1=1}^{n-2}\frac{n-1-k+1-\sum_{i=k+1}^{n-1}a_{i}}{a_k!}\nonumber\\=& c_{(a_1,\ldots,a_{n-1})}.\nonumber
    \end{align}
\end{proof}

\begin{lemma}\label{lem; 50}
For $2\leq i\leq n$, the $i$th coefficient of $p_n$ is $n+1-i$. In particular, every natural number is a coefficient of infinitely many $p_n$s.
\end{lemma}

\begin{proof}
Let $(a_1,\ldots,a_n) = (n-1,0,\ldots, 0, \overbrace{1}^{i\textnormal{th place}}, 0, \ldots0)\in\mathcal{A}_n$. Then \begin{align}c_{(a_1,\ldots,a_n)}=&\frac{1}{(n-1)!}\prod_{k=1}^{n-1}\left(n-k+1-\sum_{l=k+1}^{n}a_l\right)\nonumber\\ = &\nonumber\frac{1}{(n-1)!}\prod_{k=1}^{i-1}\left(n-k\right)\prod_{k=i}^{n-1}\left(n-k+1\right)\nonumber\\=&n+1-i\nonumber.\end{align}
\end{proof}

\begin{corollary}
For $m\in\N$ denote by $\mathbb{P}(m)$ the set of prime divisors of $m$. Then $$\bigcup_{i=m}^{n-1}\mathbb{P}(m) = \bigcup_{(a_1,\ldots,a_n)\in\mathcal{A}_n}\mathbb{P}(c_{(a_1,\ldots,a_n)}).$$
\end{corollary}

\begin{proof}
The inclusion ``$\subseteq$" follows immediately from Lemma \ref{lem; 50}. In the other direction, recall that $c_{(n,0,\ldots,0)}=1$. Thus, it suffices to consider $(a_1,\ldots,a_n)\in\mathcal{A}_n$ such that $\sum_{i=2}^na_i \geq 1$. It follows that for every $1\leq k\leq n-1$ it holds $n-k+1-\sum_{i=k+1}^na_i\leq n-1$. The assertion follows now from Lemma \ref{lem; coef}.
\end{proof}

In Remark \ref{rem; 999} we shall use the following result to lower bound the maximal coefficients of $p_n, n\in\N$. Additionally, it provides an interpretation of OEIS A009999.

\begin{lemma}\label{lem; 521}
For $1\leq j\leq n$ it holds $$c_{(j,\underbrace{1,\ldots,1}_{n-j},0,\ldots,0)}=j^{n-j}.$$
\end{lemma}

\begin{proof}
It holds 
\begin{align}
c_{(j,1,\ldots,1,0,\ldots,0)}=&\frac{1}{j!}\prod_{k=1}^{n-1}\left(n-k+1-\sum_{i=k+1}^{n}a_{i}\right)\nonumber\\=&\frac{1}{j!}\prod_{k=1}^{n-j}\left(n-k+1-(n-j-k+1)\right)\prod_{k=n-j+1}^{n-1}\left(n-k+1\right)\nonumber\\=&j^{n-j}.
\end{align}
\end{proof}

\begin{lemma}\label{lem; 5}
It holds $\sum_{(a_1,\ldots,a_n)\in \mathcal{A}_n} c_{(a_1,\ldots,a_n)}= n!$.
\end{lemma}

\begin{proof}
The assertion follows immediately from specializing $(x_1,\ldots,x_n)\mapsto(1,\ldots,1)$ in the definition of $p_n$.
\end{proof}

In the following lemma we calculate the sum of the coefficients of the monomials of $p_n$ that contain $x_i$:

\begin{lemma}\label{lem; 123}
Let $1\leq i\leq n$. Then $$\sum_{\substack{(a_1,\ldots,a_n)\in \mathcal{A}_n \\ a_i>0}}c_{(a_1,\ldots,a_n)}=(n-i+1)(n-1)!.$$ 
\end{lemma}

\begin{proof}
Specializing, $f_n(x_i):=p_n(1,\ldots,1,x_i,1,\ldots,1)$ is a (univariate) polynomial in $x_i$ whose free coefficient $r$ is the sum of the coefficients corresponding to the monomials that do not contain $x_i$. Thus, \begin{align}r=&f_n(0)\nonumber \\=&p_n(1,\ldots,1,x_i,1,\ldots,1)|_{x_i=0}\nonumber\\=&p_n(1,\ldots,1,0,1,\ldots,1)\nonumber\\=&1\cdot2\cdots(i-1)(i-1)i\cdots(n-1)\nonumber\\=&(i-1)(n-1)!.\label{eq; 56}\end{align} It follows that

\begin{align}
\sum_{\substack{(a_1,\ldots,a_n)\in \mathcal{A}_n \\ a_i>0}}c_{(a_1,\ldots,a_n)} =&\sum_{(a_1,\ldots,a_n)\in \mathcal{A}_n}c_{(a_1,\ldots,a_n)} -\sum_{\substack{(a_1,\ldots,a_n)\in \mathcal{A}_n \\ a_i=0}}c_{(a_1,\ldots,a_n)}\nonumber\\
=&n!-r=(n-i+1)(n-1)!\nonumber
\end{align}
where the second equality is due to Lemma \ref{lem; 5} and (\ref{eq; 56}).
\end{proof}

\begin{remark}
Lemma \ref{lem; 123} provides an additional interpretation of the following well known sequences:
\begin{center}
\begin{tabular}{ |c|c| } 
 \hline
 \textbf{The sequence} & \textbf{OEIS} \\ 
 \hline
 $n!$ & A000142 \\ 
 \hline
 $(n-1)(n-1)!$ & A001563\\
  \hline
 $(n-2)(n-1)!$ & A062119\\
 \hline
 $(n-3)(n-1)!$ & A052571\\
 \hline
\end{tabular}
\end{center}
\end{remark}

Combining Lemma \ref{lem; 1} together with Lemma \ref{lem; 5} we immediately obtain

\begin{corollary}\label{cor; 1}
The average of the coefficients of $p_n$ is $\frac{n!}{C_n}$.
\end{corollary}

\begin{remark}
The sequence $\left(\frac{n!}{C_n}\right)_{n\in\N}$ is rational and the sequences of the corresponding numerators and denominators are OEIS A144187 and OEIS A144186, respectively. Corollary \ref{cor; 1} provides an additional interpretation of these sequences that correspond to the denominator and enumerator of the series expansion of the EGF for the Catalan numbers, respectively.
\end{remark}

In the following lemma we calculate the sum of the coefficients of $p_n$ whose corresponding monomials have $x_i$ as a variable with maximal index. This result seems to provides the first interpretation of OEIS A299504. These monomials are being used in the proof of \cite[Lemma 3.6]{fried2021restrictiveness}. 

\begin{lemma}
Let $1\leq i\leq n$. Then $$\sum_{\substack{(a_1,\ldots,a_n)\in \mathcal{A}_n \\ a_{i+1}=\cdots=a_n=0}}c_{(a_1,\ldots,a_n)}=i!i^{n-i}.$$ 
\end{lemma}

\begin{proof}
Clearly, 
$$\sum_{\substack{(a_1,\ldots,a_n)\in \mathcal{A}_n \\ a_{i+1}=\cdots=a_n=0}}=p_n(\overbrace{1,\ldots,1}^{i \textnormal{ times }} ,\overbrace{0,\ldots,0}^{n-i\textnormal{ times}}) = i!i^{n-i}.$$
\end{proof}

\iffalse
\prod_{k=1}^{n-1}\binom{k+1-\sum_{j=n-k+1}^{n}i_{j}}{i_{n-k}}	=\prod_{k=1}^{n-1}\binom{n+1-k-\sum_{j=k+1}^{n}i_{j}}{i_{k}}=\\\binom{n-\sum_{j=2}^{n}i_{j}}{i_{1}}\binom{n-1-\sum_{j=3}^{n}i_{j}}{i_{2}}\binom{n-2-\sum_{j=4}^{n}i_{j}}{i_{3}}\cdots\binom{2-i_{n}}{i_{n-1}}=\\\frac{(n-\sum_{j=2}^{n}i_{j})!}{i_{1}!(n-\sum_{j=1}^{n}i_{j})!}\frac{(n-1-\sum_{j=3}^{n}i_{j})!}{i_{2}!(n-1-\sum_{j=2}^{n}i_{j})!}\frac{(n-2-\sum_{j=4}^{n}i_{j})!}{i_{3}!(n-2-\sum_{j=3}^{n}i_{j})!}\cdots\frac{(2-\sum_{j=n}^{n}i_{j})!}{i_{n-1}!(2-\sum_{j=n-1}^{n}i_{j})!}
	=\\ \prod_{k=1}^{n-1}\frac{(n-(k-1)-\sum_{j=k+1}^{n}i_{j})}{i_{k}!}
	=\prod_{k=1}^{n-1}\frac{\left(1-\sum_{j=k+1}^{n}(i_{j}-1)\right)}{i_{k}!}
	=\prod_{k=1}^{n-1}\frac{1+\sum_{j=1}^{k}i'_{j}}{(i'_{k}+1)!}.\\
	\prod_{k=1}^{n-1}\frac{(i'_{k}+1)+\sum_{j=1}^{k-1}i'_{j}}{(i'_{k}+1)!}\sum_{k=1}^{n}i'_{k}=0\iff-\sum_{j=k+1}^{n}i'_{j}=\sum_{j=1}^{k}i'_{j}\\ x_{1}^{i_{1}}x_{2}^{i_{2}}\cdots x_{n}^{i_{n}},\;\;\sum_{k=1}^{n}i_{k}=n,\;\;0\leq i_{k}\leq n+1-k\overset{i_{k}'=i_{k}-1}{\iff}\sum_{k=1}^{n}i'_{k}=0,\;\;-1\leq i'_{k}\leq n-k
\end{proof}

$$
\left((i'_{l+1}+1)+\sum_{j=1}^{l}i'_{j}\right)\prod_{k=l+2}^{n-1}\left((i'_{k}+1)+\sum_{j=1}^{k-1}i'_{j}\right)=\left((i'_{l+1}+1)+\sum_{j=1}^{l}i'_{j}\right)\prod_{k=l+2}^{n-1}\left((i'_{k}+1)+\sum_{j=1}^{k-1}i'_{j}\right)i'_{1}\geq\cdots\geq i'_{l}<i'_{l+1}\geq\cdots\geq i'_{n}i'_{1}\geq\cdots\geq i'_{l}+1,i'_{l+1}-1\geq\cdots\geq i'_{n}\frac{1}{(i'_{l}+2)!}\geq\frac{1}{(i'_{l+1}+1)}$$
\fi

\subsection{The maximal coefficients of $p_n, n\in\N$}

The multinomial theorem (e.g. \cite[Theorem 3.9]{detemple2014combinatorial}) states that for every $m\in\N$ it holds \begin{equation}\label{eq; 6}
(x_1+\cdots+x_n)^m=\sum_{}\binom{m}{k_1,\ldots,k_n}x_1^{k_1}\cdots x_n^{k_n}\end{equation} where the sum is over all nonnegative integers $k_1,\ldots,k_n$ such that $k_1 +\cdots+k_n = m$. It seems to be folklore that the maximal coefficient on the right-hand side of (\ref{eq; 6}) is obtained whenever $r$ of the $k_1,\ldots,k_n$ are equal to $q+1$ and the rest are equal to $q$ where $m = qn + r, \;q\in\N_0,\;0\leq r<n$. In this section we address the analogue problem of finding the maximal coefficients of $p_n, n\in\N$, the sequence of which we denote by $(m_n)_{n\in\N}$ (OEIS A349404). For example, $$(m_n)_{n\in\N} = 1, 1, 2, 4, 9, 27, 96, 384, 1536,\ldots$$ The quotients of consecutive elements of $(m_n)_{n\in\N}$ exhibit a nontrivial pattern and the induced sequence is denoted by $(q_n)_{n\in\N}$, i.e. $q_n = \frac{m_{n+1}}{m_n}, n\in\N$. Table \ref{table:1} lists $(m_n)_{n=1}^{29}$ and $(q_n)_{n=1}^{28}$ which were established by brute force. The last column in the table lists a monomial of $p_n$ whose coefficient is $m_n$. Notice that such a monomial is, in general, not unique and in Table \ref{table:2} we list all other monomials of $p_1,\ldots,p_{29}$ whose coefficients are maximal.

\bigskip

The problem of finding a maximal coefficient of $p_n$ may be formulated as a combinatorial optimization problem as follows: 
\begin{align}
    \text{maximize } & \prod_{k=1}^{n-1}\frac{n-k+1-\sum_{i=k+1}^{n}a_{i}}{a_{k}!}\nonumber\\
    \text{subject to } & (a_1,\ldots,a_n)\in\mathcal{A}_n.\nonumber
\end{align} This problem is intractable already for small values of $n$ since $C_n$ is asymptotically  $\frac{4^n}{\sqrt{\pi}n^{3/2}}$ (e.g. \cite[Problem 12-4]{cormen2009introduction}). The following two lemmas reduce the complexity of the problem. More precisely, in Lemma \ref{lem; 66} we show that it suffices to perform the search over elements of $(a_1,\ldots,a_n)\in\mathcal{A}_n$ such that $a_1\geq a_2 \geq\cdots \geq a_n$ and in Lemma \ref{lem; 67} we show that the search may be performed over $(a_1,\ldots,a_n) \in\mathcal{A}_n$ with consecutive differences bounded by $1$, i.e., $a_i-a_{i+1}\leq 1$ for every $1\leq i\leq n-1$. In Lemma \ref{lem; 199} we show that the subset $\mathcal{M}_n$ of $\mathcal{A}_n$ consisting of elements that have both properties has the same cardinality as the set of all partitions of $n$ with distinct parts, reducing the complexity of the problem to the order of $\frac{3^{3/4}}{12n^{3/4}}e^{\pi\sqrt{n/3}}$.

\begin{lemma}\label{lem; 66}
Let $(a_1,\ldots,a_n)\in \mathcal{A}_n$ such that for some $1\leq i\leq n-1$ it holds $a_i< a_{i+1}$. Let $(a'_1,\ldots,a'_n)$ be defined by $$a'_k = \begin{cases} a_k & k\neq i, i+1\\
a_{i+1} & k = i\\
a_i & k=i+1,
\end{cases} \;\;\;\forall 1\leq k\leq n.$$ Then $(a'_1,\ldots,a'_n)\in \mathcal{A}_n$ and $c_{(a'_1,\ldots,a'_n)}> c_{(a_1,\ldots,a_n)}$.
\end{lemma}
\begin{proof}
It holds
\begin{enumerate}
    \item $\prod_{k=1}^n a_k! = \prod_{k=1}^n a'_k!$,
    \item $n-k+1-\sum_{j=k+1}^{n}a_j = n-k+1-\sum_{j=k+1}^{n}a'_j$ for every $i+1\leq k\leq n$ and
    \item $n-k+1-\sum_{j=k+1}^{n}a_j = n-k+1-\sum_{j=k+1}^{n}a'_j$ for every $1\leq k\leq i-1$.
\end{enumerate} It follows from Lemma \ref{lem; coef} that \begin{align}c_{(a'_1,\ldots,a'_n)}> c_{(a_1,\ldots,a_n)}\iff& n-(i+1)+1-\sum_{j=i+1}^{n}a'_j > n-(i+1)+1-\sum_{j=i+1}^{n}a_j\nonumber\\\iff &-a'_{i+1}-\sum_{j=i+2}^{n}a'_j >-\sum_{j=i+1}^{n}a_j\nonumber\\ \iff &a_i< a_{i+1}.\nonumber\end{align}
\end{proof}

\begin{lemma}\label{lem; 67}
Let $(a_1,\ldots,a_n)\in \mathcal{A}_n$ such that for some $1\leq i\leq n-1$ it holds $a_i> a_{i+1} + 1$. Let $(a'_1,\ldots,a'_n)$ be defined by $$a'_k = \begin{cases} a_k & k\neq i, i+1\\
a_i - 1 & k = i\\
a_{i+1}+1 & k=i+1,
\end{cases} \;\;\;\forall 1\leq k\leq n.$$ Then $(a'_1,\ldots,a'_n)\in \mathcal{A}_n$ and $c_{(a'_1,\ldots,a'_n)}\geq c_{(a_1,\ldots,\mathcal{A}_n)}$.
\end{lemma}

\begin{proof}
Arguing as in the proof of Lemma \ref{lem; 66}, we have
\begin{align}c_{(a'_1,\ldots,a'_n)}\geq c_{(a_1,\ldots,a_n)}\iff & \frac{n-i-\sum_{j=i+1}^{n}a{}_{j}}{a_{i+1}+1}\geq\frac{n-i-\sum_{j=i+1}^{n}a{}_{j}+1}{a_{i}}\nonumber \\\iff &n-i-\sum_{j=i+1}^{n}a{}_{j}\geq\frac{a_{i+1}+1}{a_{i}-(a_{i+1}+1)}\nonumber\\\iff&n-(i-1)-\sum_{j=i+1}^{n}a{}_{j}\geq\frac{a_i}{a_{i}-(a_{i+1}+1)}\nonumber. \end{align}
Since $a_{i}-(a_{i+1}+1)\geq 1$ it suffices to show that $n-(i-1)-\sum_{j=i+1}^{n}a{}_{j}\geq a_i$ or, equivalently, that $\sum_{j=i}^{n}a{}_{j}\leq n-(i-1)$ but this holds by the definition of $\mathcal{A}_n$.
\end{proof}

Combining Lemma \ref{lem; 66} and Lemma \ref{lem; 67} we obtain

\begin{theorem}\label{thm; 1}
A maximal coefficient of $p_n$ is attained at a %monomial belonging to $\mathcal{M}_n$ where $$\mathcal{M}_n=\{(a_1,\ldots,a_n)\in \mathcal{A}_n\;|\;a_1\geq\cdots\geq a_n,\; a_i-a_{i+1}\leq 1, \;\forall 1\leq i\leq n-1\}.$$ 
monomial belonging to $\mathcal{M}_n$ where $$\mathcal{M}_n=\{(a_1,\ldots,a_n)\in \mathcal{A}_n\;|\;a_{i+1}\leq a_i\leq a_{i+1}+1, \;\forall 1\leq i\leq n-1\}.$$ 
\end{theorem}

\begin{lemma}\label{lem; 199}
\begin{enumerate}
    \item There is a bijection between the set $$\{(a_1,\ldots,a_n)\in\mathcal{A}_n\;|\;a_1\geq a_2\geq\cdots\geq a_n\}$$ considered in Lemma \ref{lem; 66} and the set of all partitions of $n$. In particular, Lemma \ref{lem; 66} reduces the complexity of the problem to the order of $\frac{1}{4n\sqrt{3}}e^{\pi\sqrt{2n/3}}$.
    %A000041.
    \item There is a bijection between $\mathcal{M}_n$ and the set of all partitions of $n$ with distinct parts. In particular, Theorem \ref{thm; 1} reduces the complexity of the problem to the order of $\frac{3^{3/4}}{12n^{3/4}}e^{\pi\sqrt{n/3}}$.
\end{enumerate}
\end{lemma}

\begin{proof} Recall that, by definition (e.g. \cite[Definition 1.1]{andrews_1984}), a partition of $n$ is a nonincreasing sequence $a_1\geq a_2\geq \cdots \geq a_r$ of $r\in\N$ natural numbers such that $\sum_{i=1}^r a_i = n$. One also writes $(a_1,\ldots,a_r)$ for such a partition.
\begin{enumerate}
    \item Let $(a_1,\ldots,a_n)\in\mathcal{A}_n$ such that $a_1\geq a_2\geq\cdots\geq a_n$ and let $r = \max\{1\leq i\leq n\;|\; a_i>0\}$. Then $(a_1,\ldots,a_r)$ is a partition of $n$. Conversely, let $a_1,\ldots,a_r\in\N$ be a partition of $n$ with $r\in\N$ parts. Extend $(a_1,\ldots,a_r)$ with $n-r$ zeros to obtain $(a_1,\ldots,a_r, 0,\ldots,0)\in\N_0^n$. Let $1\leq l\leq r-1$ and suppose $\sum_{i=l+1}^na_i>n-l$. Since $a_r\geq 1$, by monotonicity, also $a_1,\ldots,a_{r-1}\geq 1$. Thus, $$n=\sum_{i=1}^n a_i =\sum_{i=1}^l a_i+\sum_{i=l+1}^n a_i>l + n-l=n.$$ A contradiction. This shows that $(a_1,\ldots,a_r, 0,\ldots,0)\in\mathcal{A}_n$.
    The claim regarding the complexity is due to \cite{hardy1918asymptotic}, \cite{uspensky1920asymptotic} and \cite{rademacher1938partition}.
    \item Let $(a_1,\ldots,a_n)\in\mathcal{M}_n$ and let  $r = \max\{1\leq i\leq n\;|\; a_i>0\}$. Then $(a_1,\ldots,a_r)$ is a partition of $n$ such that $a_i-a_{i+1}\leq 1$ for every $1\leq i\leq r-1$ and $a_r=1$. This means that each of the numbers $1,2,\ldots, a_1 =: m$ appears as part in $(a_1,\ldots,a_r)$. We claim that the conjugate $(a'_1,\ldots,a'_m)$ of $(a_1,\ldots,a_r)$ (cf. \cite[Definition 1.8]{andrews_1984}) has distinct parts. To see that, recall that for $1\leq i\leq m,\;$ $a'_i$ is defined to be the number of parts of $(a_1,\ldots,a_r)$ that are $\geq i$. Suppose $a'_k = a'_{k+1}$ for some $1\leq k\leq m-1$. This means that every $1\leq i\leq r$ if $a_i\geq k$ then also $a_i\geq k+1$. Thus, $k$ cannot be a part of $(a_1,\ldots,a_r)$, a contradiction. 
    
    Conversely, let $(a'_1,\ldots,a'_m)$ be a partition of $n$ with $1\leq m\leq n$ distinct parts and let $(a_1,\ldots,a_r)$ be its conjugate where $1\leq r\leq n$. Suppose that for some $1\leq i\leq m-1$ it holds $a_i\geq a_{i+1}+2$. Thus, there are $1\leq k,l \leq m$ such that $k\neq l$ and $i\leq a'_k,a'_l<i+1$. It follows that $a'_k=i=a'_l$, a contradiction. It remains to show that $a_r=1$. Since the parts of $(a'_1,\ldots,a'_m)$ are distinct, $a'_1>a'_i$ for every $2\leq i\leq m$. Thus, $r = a'_1$ and $a_r=1$ since $a_r$ is the number of parts of $(a'_1,\ldots,a'_m)$ that are $\geq a'_1$.
    
    The claim regarding the complexity may be found, for example, in \cite[(37) on p. 45]{flajolet2009analytic}.
\end{enumerate}
\end{proof}

\begin{remark}
The idea of the proof of the second statement in Lemma \ref{lem; 199} is due to J. Grahl and F. T. Adams-Watters (cf. comments to OEIS A000009). 
\end{remark}

\bigskip

Despite the complexity reduction described above, the problem remains intractable. We provide a simple algorithm %algorithm with complexity in the order of $O(n^2)$ 
that for a prescribed $l\in \N$ successively generates a sequence of monomials $(r_n)_{n=1}^l$ such that $r_n\in\mathcal{A}_n$ for every $1\leq n\leq l$. In Lemma \ref{lem; 888} we show that actually $r_n\in\mathcal{M}_n$. The algorithm also returns a sequence $(s_n)_{n=1}^l$ of the corresponding coefficients. It uses a method $\textnormal{Coefficient}(\cdot)$ that upon receiving a monomial of $p_n$ as input returns its corresponding coefficient, e.g., by using the formula given by Lemma \ref{lem; coef}. We applied the algorithm with $l=100$ and Table \ref{table:3} lists the elements of $(s_n)_{n=29}^{100}$ returned by the algorithm (the first $29$ elements of $(s_n)_{n=1}^{100}$ coincide with $(m_n)_{n=1}^{29}$ that were already listed in Table \ref{table:1}). We consider the correctness of the first $29$ numbers returned by the algorithm and the preservation of the pattern that the elements of $(q_n)_{n=1}^{28}$ exhibit (this is illustrated in the third column of Table \ref{table:3}) to be a strong indication that the algorithm actually returns the true sequence $(m_n)_{n\in\N}$. Before we present the pseudocode of the algorithm, let us illustrate by an example its idea: Assume that the algorithm decided that $r_{15} = (3, 3, 2, 2, 1, 1, 1, 1, 1, 0, 0, 0, 0, 0, 0)$. First, a zero is attached at the end of $r_{15}$ and the result is denoted by $r'_{16}$. Thus, $r'_{16}=(3, 3, 2, 2, 1, 1, 1, 1, 1, 0, 0, 0, 0, 0, 0,0)$. Now, beginning at the end of $r'_{16}$, we iterate backwards and every time\footnote{Actually, every time is wasteful and we also increase the first entry by $1$ if it does not result in a gap of $2$ between the first and the second entries.} the value is about to increase, we add $1$ at this point. Those are the candidates to pick $r_{16}$ from. Thus, we get $$\scriptstyle (3, 3, 2, 2, 1, 1, 1, 1, 1, \mathbf{1}, 0, 0, 0, 0, 0, 0), \;(3, 3, 2, 2, \mathbf{2}, 1, 1, 1, 1, 0, 0, 0, 0, 0, 0, 0), \;(3, 3, \mathbf{3}, 2, 1, 1, 1, 1, 1, 0, 0, 0, 0, 0, 0, 0), \;(\mathbf{4}, 3, 2, 2, 1, 1, 1, 1, 1, 0, 0, 0, 0, 0, 0, 0).$$ For each of the candidates we calculate the corresponding coefficient. Thus, we obtain $$370594350, \;361267200, \;321126400, \;48168960.$$ We take $r_{16}$ to be the monomial with the largest coefficient. In this case, $$r_{16} = (3, 3, 2, 2, 1, 1, 1, 1, 1, 1, 0, 0, 0, 0, 0, 0).$$ Now, we repeat the procedure described above with $r_{16}$.

\bigskip

\begin{algorithm}[H]
\caption{Generating a sequence of monomials \label{IR}}
 \KwIn{The length $l$ of the desired sequence}
 \KwOut{A sequence $(r_n)_{n=1}^l$ of monomials and the sequence $(s_n)_{n=1}^l$ of their corresponding coefficients}
 $r_1 \leftarrow (1)$\\
 $s_1 \leftarrow 1$\\
 \For{$n \leftarrow 2$ \KwTo $l$}{
   $r'_n  \leftarrow r_{n-1}\cup 0$\\
   $R  \leftarrow \{\}$\\
   $S  \leftarrow \{\}$\\
   
   \For{$i \leftarrow n$ \KwTo $1$}{
   \If{$(i=n$ {\bf and} $r'_n[n] \neq r'_n[n-1])$ {\bf or}\\ $(i> 1$ {\bf and} $r'_n[i] \neq r'_n[i-1]$ {\bf and} $r'_n[i] = r'_n[i+1] )$ {\bf or}\\ $(i=1$ {\bf and} $r'_n[1] = r'_n[2])$  }{
			 $\text{temp} \leftarrow r'_n$\\
			 $\text{temp}[i] \leftarrow \text{temp}[i]+1$\\
			 $R  \leftarrow R\cup\text{temp}$\\
			 $S  \leftarrow S\cup \text{Coefficient}(\text{temp})$\\
		}
   }
   $k \leftarrow \argmax S$\\
   $r_n \leftarrow R[k]$\\
   $s_n \leftarrow S[k]$\\   
 }
 \KwRet $(r_n)_{n=1}^l$, $(s_n)_{n=1}^l$
\end{algorithm}

\begin{lemma}\label{lem; 888}
Let $l\in\N$ and let $(r_n)_{n=1}^l$ be the sequence of monomials returned by Algorithm  \ref{IR}. Then $r_n\in \mathcal{M}_n $ for every $1\leq n\leq l$. 
\end{lemma}

\begin{proof}
First notice that for each $2\leq n\leq l$ the condition in the inner for loop is satisfied at least once. Indeed, if the first two conditions are not satisfied, then the third must be. Now, for $n=1$ it holds $r_1 = (1)\in \mathcal{M}_1$. Assume that $r_n=(a_1,\ldots,a_n)\in \mathcal{M}_n$. The monomial $r_{n+1}$ is chosen from the set $R$ and therefore it suffices to show that $x\in \mathcal{M}_{n+1}$ for every $x\in R$. By definition of $\mathcal{M}_n$, $$a_1\geq\cdots\geq a_n \text{ and } a_i-a_{i+1}\leq 1, \;\forall 1\leq i\leq n-1.$$ We give the details only in case that $x\in R$ due to $i>1, a_{i}\neq a_{i-1}$ and $a_i=a_{i+1}$, the other two cases being similar. Let $a_{n+1}=0$. First, notice that $a_n\in\{0,1\}$. Thus, $a_n\geq a_{n+1}$ and $a_n-a_{n+1}\leq 1$. Thus, $x=(a_1,\ldots,a_{i-1}, a_i+1,a_{i+1},\ldots,a_{n+1})\in \mathcal{A}_n$. By the assumptions, $a_{i-1} = a_i+1$ and $a_i+1-a_{i+1}=1$. Thus, $x\in \mathcal{M}_n$.
\end{proof}

\begin{remark}\label{rem; 999}
Recall that in Lemma \ref{lem; 521} we have seen that for every $n\in\N$ and $1\leq j\leq n$ there is a coefficients of $p_n$ that is equal to $j^{n-j}$. Elementary differentiation shows (e.g. \cite[Theorem 3.3]{gray2008largest}) that \begin{equation}\label{eq; 889} \max\{j^{n-j}\;|\;1\leq j\leq n\}=\max\{t^{n-t}\;|\;t\in\{\left\lfloor u\right\rfloor ,\left\lceil u \right\rceil \} \}\end{equation} where $u$ is a solution to the equation $x(1+\ln x)=n,\;x\in\mathbb{R}$. The left-hand side of (\ref{eq; 889}) corresponds to OEIS A003320 which, therefore, provides a lower bound to $(m_n)_{n\in\N}$.

\end{remark}

\section{Open questions}
\iffalse
In this work we have studied the coefficients of the polynomials $p_n = x_1(x_1+x_2)\cdots(x_1+x_2+\cdots+x_n)$. These polynomials are directly related to the Catalan numbers which may be interpreted as counting the number of  distinct monomials in the expansion of $p_n, n\in\N$. We proved several properties of the coefficients and established a closed formula for them that depends only on the monomials to which they are attached. We then devoted our attention to the sequence of the maximal coefficients. We showed that for every $n\in\N$ the maximum $m_n$ can only be attained if the corresponding monomial belongs to a family $\mathcal{M}_n$ of monomials that is equivalent to the classical set of partitions with odd parts. Finally, we provided an algorithm that generates for every $n\in \N$ a monomial $r_n\in \mathcal{M}_n$, inducing a sequence of lower bounds to $(m_n)_{n\in\N}$.\fi 

The following questions remain open and are left for future research:

\begin{enumerate}
    \item Is there a closed formula for the maximal coefficients $(m_n)_{n\in\N}$?
    \item Prove or disprove: The sequence $(s_n)_{n\in\N}$ returned by Algorithm \ref{IR} is equal to $(m_n)_{n\in\N}$.
    \item The sequence $(q_n)_{n\in\N}$ exhibits a nontrivial pattern. Can we predict for which $n\in\N$ it holds that $q_n\in\N$? What is the reason that some natural numbers seem to be missing from $(q_n)_{n\in\N}$? According to the output of  Algorithm \ref{IR}, for $n\leq 200$, those numbers are $15, 51, 54$ and $73$.
    \item For $n<30$ the maximal coefficient is not uniquely attained only for $n=2,5,6,12,13,14$ and $15$. Are those the only cases when this happens? If not, can we predict when?
\end{enumerate}

\bibliography{bibliography}
\bibliographystyle{plain}

\newpage
\begin{table}
\centering
\vspace*{1mm}
\caption{The maximal coefficients of $p_n$ and their consecutive quotients. The last column lists a monomial of $p_n$ whose coefficient is equal to $m_n$}
\label{table:1}
\tiny
\renewcommand{\arraystretch}{1.50}
\begin{tabular}{ |c|c|c|l| } 
 \hline
 \multicolumn{1}{|c|}{$n$} & \multicolumn{1}{c|}{$m_n$} &\multicolumn{1}{c|}{$q_n$}&\multicolumn{1}{c|}{$(a_1,\ldots,a_n)$}\\
 \hline
1 &1&1 & (1)\\
 \hline
2 &1& 2& (1, 1)\\
 \hline
3& 2& 2& (2, 1, 0)\\
 \hline
4& 4& $\frac{9}{4}$& (2, 1, 1, 0)\\
 \hline
5& 9& 3&  (2, 2, 1, 0, 0)\\
 \hline
6& 27& $\frac{32}{9}$&  (2, 2, 1, 1, 0, 0)\\
 \hline
7& 96& 4& (3, 2, 1, 1, 0, 0, 0)\\
 \hline
8& 384& 4& (3, 2, 1, 1, 1, 0, 0, 0)\\
 \hline
9& 1536& $\frac{625}{128}$& (3, 2, 1, 1, 1, 1, 0, 0, 0)\\
 \hline
10& 7500& 5& (3, 2, 2, 1, 1, 1, 0, 0, 0, 0)\\
 \hline
11& 37500& $\frac{648}{125}$& (3, 2, 2, 1, 1, 1, 1, 0, 0, 0, 0)\\
 \hline
12& 194400& 6&  (3, 2, 2, 2, 1, 1, 1, 0, 0, 0, 0, 0)\\
 \hline
13& 1166400& $\frac{16807}{2592}$& (3, 2, 2, 2, 1, 1, 1, 1, 0, 0, 0, 0, 0)\\
 \hline
14& 7563150& 7&  (3, 3, 2, 2, 1, 1, 1, 1, 0, 0, 0, 0, 0, 0)\\
 \hline
15& 52942050& $\frac{262144}{36015}$& (3, 3, 2, 2, 1, 1, 1, 1, 1, 0, 0, 0, 0, 0, 0)\\ 
 \hline
16& 385351680& 8& (4, 3, 2, 2, 1, 1, 1, 1, 1, 0, 0, 0, 0, 0, 0, 0)\\
 \hline
17& 3082813440& $\frac{531441}{65536}$& (4, 3, 2, 2, 1, 1, 1, 1, 1, 1, 0, 0, 0, 0, 0, 0, 0)\\
 \hline
18& 24998984640& 9& (4, 3, 2, 2, 2, 1, 1, 1, 1, 1, 0, 0, 0, 0, 0, 0, 0, 0)\\
 \hline
19& 224990861760& 9& (4, 3, 2, 2, 2, 1, 1, 1, 1, 1, 1, 0, 0, 0, 0, 0, 0, 0, 0)\\
 \hline
20& 2024917755840 &$\frac{5000000}{531441}$& (4, 3, 2, 2, 2, 1, 1, 1, 1, 1, 1, 1, 0, 0, 0, 0, 0, 0, 0, 0)\\
 \hline
21& 19051200000000& 10& (4, 3, 2, 2, 2, 2, 1, 1, 1, 1, 1, 1, 0, 0, 0, 0, 0, 0, 0, 0, 0)\\
 \hline
22& 190512000000000& $\frac{214358881}{21000000}$ &(4, 3, 2, 2, 2, 2, 1, 1, 1, 1, 1, 1, 1, 0, 0, 0, 0, 0, 0, 0, 0, 0)\\
 \hline
23& 1944663768432000& 11& (4, 3, 3, 2, 2, 2, 1, 1, 1, 1, 1, 1, 1, 0, 0, 0, 0, 0, 0, 0, 0, 0, 0)\\
 \hline
24& 21391301452752000& $\frac{214990848}{19487171}$& (4, 3, 3, 2, 2, 2, 1, 1, 1, 1, 1, 1, 1, 1, 0, 0, 0, 0, 0, 0, 0, 0, 0, 0)\\
 \hline
25& 235998033739776000& 12& (4, 3, 3, 2, 2, 2, 2, 1, 1, 1, 1, 1, 1, 1, 0, 0, 0, 0, 0, 0, 0, 0, 0, 0, 0)\\
 \hline
26& 2831976404877312000& 12& (4, 3, 3, 2, 2, 2, 2, 1, 1, 1, 1, 1, 1, 1, 1, 0, 0, 0, 0, 0, 0, 0, 0, 0, 0, 0)\\
 \hline
27& 33983716858527744000& $\frac{10604499373}{859963392}$& (4, 3, 3, 2, 2, 2, 2, 1, 1, 1, 1, 1, 1, 1, 1, 1, 0, 0, 0, 0, 0, 0, 0, 0, 0, 0, 0)\\
 \hline
28& 419064703766444736000& 13 &(4, 3, 3, 2, 2, 2, 2, 2, 1, 1, 1, 1, 1, 1, 1, 1, 0, 0, 0, 0, 0, 0, 0, 0, 0, 0, 0, 0)\\
 \hline
29& 5447841148963781568000& & (4, 3, 3, 2, 2, 2, 2, 2, 1, 1, 1, 1, 1, 1, 1, 1, 1, 0, 0, 0, 0, 0, 0, 0, 0, 0, 0, 0, 0)\\
 \hline
\end{tabular}
\end{table}

\begin{table}
\centering
\vspace*{1mm}
\caption{Additional monomials of $p_n$ whose coefficients are equal to $m_n$ for $n<30$}
\label{table:2}
\tiny
\renewcommand{\arraystretch}{1.50}
\begin{tabular}{ |c|l| } 
 \hline
   \multicolumn{1}{|c|}{$n$} &\multicolumn{1}{c|}{ $(a_1,\ldots,a_n)$}\\
 %$n$ & $(a_1,\ldots,a_n)$\\ 
 \hline
2 & (2, 0)\\
 \hline
5 &(3, 1, 1, 0, 0)\\
 \hline
 6 & (3, 1, 1, 1, 0, 0)\\
 \hline
 12&(3, 3, 2, 1, 1, 1, 1, 0, 0, 0, 0, 0), (4, 2, 2, 1, 1, 1, 1, 0, 0, 0, 0, 0)\\
 \hline
13& (4, 2, 2, 1, 1, 1, 1, 1, 0, 0, 0, 0, 0), (3, 3, 2, 1, 1, 1, 1, 1, 0, 0, 0, 0, 0)\\
 \hline
 14&(4, 2, 2, 2, 1, 1, 1, 1, 0, 0, 0, 0, 0, 0)\\
 \hline
15&(4, 2, 2, 2, 1, 1, 1, 1, 1, 0, 0, 0, 0, 0, 0)\\
 \hline
\end{tabular}
\end{table}
\newpage

\begin{table}
\centering
\vspace*{1mm}
\caption{$s_{30},\ldots,s_{100}$ returned by Algorithm \ref{IR} and their consecutive quotients}
\label{table:3}
\fontsize{4}{4}\selectfont
\renewcommand{\arraystretch}{2.30}
\begin{tabular}{ |c|c|c| } 
\hline
  \multicolumn{1}{|c|}{$n$} & \multicolumn{1}{c|}{$s_n$} &\multicolumn{1}{c|}{$\frac{s_{n+1}}{s_n}$}\\
\hline
30 &5447841148963781568000  & $\frac{ 289254654976 }{ 22024729467 }$\\
\hline
31 &71547458245452693504000  & 14 \\
\hline
32 &1001664415436337709056000 & 14 \\
\hline
33 &14023301816108727926784000   & $\frac{ 8649755859375 }{ 578509309952 }$\\
\hline
34 &209673612792393750000000000  & $\frac{ 962072674304 }{ 64072265625 }$\\
\hline
35 &3148339635292233982279680000 & 16 \\
\hline
36 &50373434164675743716474880000   & 16 \\
\hline
37 &805974946634811899463598080000  & $\frac{ 582622237229761 }{ 35184372088832 }$\\
\hline
38 &13346235805315454304418995840000 & 17 \\
\hline
39 &226886008690362723175122929280000  & $\frac{ 46273255257047040 }{ 2638936015687741 }$\\
\hline
40 &3978404225024620607607354163200000  & 18 \\
\hline
41 &71611276050443170936932374937600000   & $\frac{ 42052983462257059 }{ 2313662762852352 }$\\
\hline
42 &1301601882440198387257545835699200000 & 19 \\
\hline
43 &24730435766363769357893370878284800000  & 19 \\
\hline
44 &469878279560911617799974046687411200000  & $\frac{ 32768000000000000000 }{ 1640066355028025301 }$\\
\hline
45 &9388017391765133721600000000000000000000   & 20 \\
\hline
46 &187760347835302674432000000000000000000000  & $\frac{ 68122318582951682301 }{ 3276800000000000000 }$\\
\hline
47 &3903402780908908402684872776385162240000000 & 21 \\
\hline
48 & 81971458399087076456382328304088407040000000& $\frac{ 68440034007706025984 }{ 3243919932521508681 }$\\
\hline
49 & 1729429060270920090607276147799029186560000000  & 22 \\
\hline
50 & 38047439325960241993360075251578642104320000000 & 22 \\
\hline
51 & 837043665171125323853921655534730126295040000000 & $\frac{ 141050039560662968926103 }{ 6159603060693542338560 }$\\
\hline
52 & 19167638064180062036564117435227926126077952000000 & 23 \\
\hline
53 & 440855675476141426840974701010242300899792896000000 & $\frac{ 145398897491341278707712 }{ 6132610415680998648961 }$\\
\hline
54 & 10452307389872489948744789613053673200242655232000000  & 24 \\
\hline
55 &250855377356939758769874950713288156805823725568000000  & $\frac{ 582076609134674072265625 }{ 24233149581890213117952 }$\\
\hline
56 &6025508444195271521816253662109375000000000000000000000  & 25 \\
\hline
57 &150637711104881788045406341552734375000000000000000000000  & $\frac{ 4116767537697256247666432 }{ 162981450557708740234375 }$\\
\hline
58 &3804975577941702545004131009801025732678899466240000000000  & 26 \\
\hline
59 &98929365026484266170107406254826669049651386122240000000000  & 26 \\
\hline
60 &2572163490688590920422792562625493395290936039178240000000000  & $\frac{ 1570042899082081611640534563 }{ 58959020400027837728817152 }$\\
\hline
61 &68495151317539451014974372699807632703067401288427560000000000  & 27 \\
\hline
62 &1849369085573565177404308062894806082982819834787544120000000000  & $\frac{ 43866262300411718040591269888 }{ 1570042899082081611640534563 }$\\
\hline
63 &51670504955929472544522433459271401555552182671108997120000000000  & 28 \\
\hline
64 &1446774138766025231246628136859599243555461114791051919360000000000  & $\frac{ 176994576151109753197786640401 }{ 6266608900058816862941609984 }$\\
\hline
65 &40862798295083655283632061806834793855899184040599487771040000000000  & 29 \\
\hline
66 &1185021150557426003225329792398209021821076337177385145360160000000000  & 29 \\
\hline
67 &34365613366165354093534563979548061632811213778144169215444640000000000  & $\frac{ 5230176601500000000000000000000 }{ 176994576151109753197786640401 }$\\
\hline
68 &1015501326834227574804448936477849769406960000000000000000000000000000000  & 30 \\
\hline
69 &30465039805026827244133468094335493082208800000000000000000000000000000000  & $\frac{ 645590698195138073036733040138561 }{ 20920706406000000000000000000000 }$\\
\hline
70 &940118652620125224647533652879072642271519651664470795263552602800000000000  & 31 \\
\hline
71 &29143678231223881964073543239251251910417109201598594653170130686800000000000  & $\frac{ 649037107316853453566312041152512 }{ 20825506393391550743120420649631 }$\\
\hline
72 &908277006976838893495396839720732975992938383907565812097329240473600000000000  & 32 \\
\hline
73 &29064864223258844591852698871063455231774028285042105987114535695155200000000000  & 32 \\
\hline
74 &930075655144283026939286363874030567416768905121347391587665142244966400000000000  & $\frac{ 2781855434090103443811378243892171521 }{ 85165837925733364110154508149981184 }$\\
\hline
75 &30379974863093055499068023870097652840716436890381792175415887077743001600000000000  & 33 \\
\hline
76 &1002539170482070831469244787713222543743642417382599141788724273565519052800000000000  & $\frac{ 2847501839779123940187735784914157568 }{ 84298649517881922539738734663399137 }$\\
\hline
77 &33864506118722234127534503334495450730573121152148488603345499618079355699200000000000  & 34 \\
\hline
78 &1151393208036555960336173113372845324839486119173048612513746987014698093772800000000000  & $\frac{ 399669593472470313551127910614013671875 }{ 11557507467338797168997280538769227776 }$\\
\hline
79 &39816271517312974256100509002896515744268277555434859054687500000000000000000000000000000  & 35 \\
\hline
80 &1393569503105954098963517815101378051049389714440220066914062500000000000000000000000000000 & $\frac{ 404140638732382030321569800228268146688 }{ 11419131242070580387175083160400390625 }$\\
\hline
81 & 49320570642735368629992208025957484343031663764994032732691353227472207179939840000000000000 & 36 \\
\hline
82 & 1775540543138473270679719488934469436349139895539785178376888716188999458477834240000000000000 & 36 \\
\hline
83 & 63919459552985037744469901601640899708569036239432266421567993782803980505202032640000000000000 & $\frac{ 59325966985223687799599734398071581327609 }{ 1616562554929528121286279200913072586752 }$\\
\hline
84 & 2345769877936484815469854363747603640274357258394513512999736861193926236361921658880000000000000 & 37 \\
\hline
85 & 86793485483649938172384611458661334690151218560596999980990263864175270745391101378560000000000000 & $\frac{ 36807814606597472834439071501847169493876736 }{ 970059730434063003209671332725224505491985 }$\\
\hline
86 & 3293280220294377971347467731349401295685328549014404380417498855020965676295207000211456000000000000 & 38 \\
\hline
87 & 125144648371186362911203773791277249236042484862547366455864956490796695699217866008035328000000000000& $\frac{ 9093778876146525519753713411306280250639479 }{ 237357799848111043775597979581466351239168 }$\\
\hline
88 & 4794608647994409038997330943549548574735166856026307562390437024839722999393985472296755584000000000000 & 39 \\
\hline
89 & 186989737271781952520895906798432394414671507385025994933227043968749196976365433419573467776000000000000  & $\frac{ 360287970189639680000000000000000000000000000 }{ 9093778876146525519753713411306280250639479 }$\\
\hline
90 & 7408378167700986380828325883329584618472057265089496740817451089920000000000000000000000000000000000000000 & 40 \\
\hline
91 & 296335126708039455233133035333183384738882290603579869632698043596800000000000000000000000000000000000000000 & 40 \\
\hline
92 & 11853405068321578209325321413327335389555291624143194785307921743872000000000000000000000000000000000000000000& $\frac{ 58983677299744401560074115672359981890669066761 }{ 1441151880758558720000000000000000000000000000 }$\\
\hline
93 & 485137915571417159462028231314994899343248266747019135549714969297846930802852529386687060879513600000000000000  & 41 \\
\hline
94 & 19890654538428103537943157483914790873073178936627784557538313741211724162916953704854169496060057600000000000000 & $\frac{ 11152795537017915251768561543888501580276393855418368 }{ 267188865036464121457189393493879454552305181077725 }$\\
\hline
95 & 830260659007146007772745689466656164663688800918177680838519253171995494099296780078316846245644075008000000000000& 42 \\
\hline
96 & 34870947678300132326455318957599558915874929638563462595217808633223810752170464763289307542317051150336000000000000 & $\frac{ 10093776109231555797740541116805209814919811290249 }{ 237279062189656720523796155594134696283911225344 }$\\
\hline
97 & 1483399063252167028338360029439804956347076714866041438488642548695961414227705961250855012767516543942656000000000000 & 43 \\
\hline
98 & 63786159719843182218549481265911613122924298739239781855011629593926340811791356333786765549003211389534208000000000000 & $\frac{ 885182809206724429753029845421970028406525755654144 }{ 20422291197747566381475048306094261718558687959341 }$\\
\hline
99 & 2764744244541327539828782867196968086879145195304124286644923210215005242351869797701380142664027856774889472000000000000 & 44 \\
\hline
100 & 121648746759818411752466446156666595822682388593381468612376621249460230663482271098860726277217225698095136768000000000000 & \\
\hline
\end{tabular}
\end{table}

\end{document}